\documentclass[11pt,leqno]{article}

\usepackage{anysize}
\marginsize{3.5cm}{3.5cm}{3cm}{3cm}

\usepackage[]{hyperref}
\hypersetup{
	colorlinks=true,       
	linkcolor=red,          
	citecolor=blue,        
	filecolor=magenta,      
	urlcolor=cyan           
}

\usepackage{amsmath}
\usepackage{amsfonts,amssymb}
\usepackage{enumerate}
\usepackage{mathrsfs}
\usepackage{amsthm}
\usepackage{xcolor}
\usepackage{tikz}
\usepackage{textcomp}
\usepackage{ucs}
\usepackage{textcomp}
\usepackage{dsfont}

\theoremstyle{plain}
\newtheorem{theo}{Theorem}[section]
\newtheorem*{theo*}{Theorem}
\newtheorem{prop}[theo]{Proposition}
\newtheorem{lemm}[theo]{Lemma}

\theoremstyle{definition}
\newtheorem{rema}[theo]{Remark}

\DeclareSymbolFont{pletters}{OT1}{cmr}{m}{sl}
\DeclareMathSymbol{s}{\mathalpha}{pletters}{`s}

\renewcommand{\S}{\mathbb{S}}
\newcommand{\abs}[1]{\left|#1\right|}
\newcommand{\pa}[1]{\left(#1\right)}

\def\ba{\begin{align}}
	\def\bad{\begin{aligned}}
		\def\be{\begin{equation}}
			\def\ea{\end{align}}
		\def\ead{\end{aligned}}
	\def\ee{\end{equation}}

\def\eps{\varepsilon}

\def\le{\leq}

\newcommand{\br}[1]{\left\{#1\right\}}
\newcommand{\norm}[1]{\left\|#1\right\|}
\newcommand{\cro}[1]{\left[#1\right]}

\def \R{\mathbb{R}}

\def \N{\mathbf{N}}

\numberwithin{equation}{section}

\title{Well posedness for systems of self-propelled particles} 
\author{Marc Briant\thanks{Map5, UMR 8145 CNRS, Universit\'e de Paris, France,  briant.maths@gmail.com} \and Nicolas Meunier\thanks{LaMME, UMR CNRS 8071, Universit\'e \'Evry Val d’Essonne, France, nicolas.meunier@univ-evry.fr}
}

\begin{document}
	
	\maketitle

\begin{abstract}
	This paper deals with the existence and uniqueness of solutions to kinetic equations describing alignment of self-propelled particles. The particularity of these models is that the velocity variable is not on the euclidean space but constrained on the unit sphere (the self-propulsion constraint). Two related equations are considered : the first one in which the alignment mechanism is nonlocal, using an observation kernel depending on the space variable, and a second form which is purely local, corresponding to the principal order of a scaling limit of the first one. We prove local existence and uniqueness of weak solutions in both cases for bounded initial conditions (in space and velocity) with finite total mass. The solution is proven to depend continuously on the initial data in $L^p$ spaces with finite p. In the case of a bounded kernel of observation, we obtain that the solution is global in time. Finally by exploiting the fact that the second equation corresponds to the principal order of a scaling limit of the first one we deduce a Cauchy theory for an approximate problem approaching the second one.
\end{abstract}



\section{Introduction}

In this work, we are interested in the study of two partial differential equations, which differ only by a spatial integral in the nonlinearity. We consider the following kinetic equation for the distribution function $f = f(t,x,\omega )$ for  $(t,x,\omega) \in \R^+\times \R^d  \times \S^{d-1}$:
\begin{equation}\label{eq:general kinetic_0}
\partial_t f + \nabla_x\cdot\pa{c\omega f} = \sigma \Delta_\omega f +\nu \nabla_\omega \cdot \pa{f \mathbf{P}_{\omega^\bot}(\mathbf{F_0}[f])}, 
\end{equation}
where $\nabla_ \omega$ , $\nabla_ \omega \cdot$  are the gradient and divergence operators on the sphere; $c > 0$, $\nu>0$ and $\sigma >0$ are positive constants and $\mathbf{F_0}$  is given by
\begin{eqnarray}
& &\mathbf{F_0}[f](t,x,v) \nonumber \\&& \qquad  = \int_{\R^d } \int_{\S^{d-1}} \mathbf{K}\pa{\abs{x-x_*},(x_*-x)\cdot \omega, (x-x_*)\cdot \omega_*} f(t,x_*,\omega_*)\:dx_*d\omega_*, \label{eq:integration kernel_0}
\end{eqnarray}
and the interaction kernel $K : \R^{d} \times   \S^{d-1} \times   \S^{d-1}\rightarrow  \R^d$ satisfies
\begin{equation}\label{def:reg_W}
K\in L^1_{x}W^{1,\infty}_\omega L^1_{\omega_*}.
\end{equation}

Denote
\begin{equation}\label{def:k}
	\mathbf{k}(\omega,\omega_*) = \int_{\R^d } \mathbf{K}\pa{\abs{x_*},x_*\cdot \omega, -x_*\cdot \omega_*}\:dx_* \in W^{1,\infty}_\omega L^1_{\omega_*}.
\end{equation}
The second kinetic equation for the distribution function $f = f(t,x,\omega )$ for  $(t,x,\omega) \in \R^+\times \R^d  \times \S^{d-1}$ is:
\begin{equation}\label{eq:general kinetic_1}
	\partial_t f + \nabla_x\cdot\pa{c\omega f} = \sigma \Delta_\omega f +\nu \nabla_\omega \cdot \pa{f \mathbf{P}_{\omega^\bot}(\mathbf{F_1}[f])},  
\end{equation}
with $\mathbf{F_1}$  given by
\begin{equation}
	\mathbf{F_1}[f](t,x,\omega) = \int_{\S^{d-1}} \mathbf{k}(\omega,\omega_*) f(t,x,\omega_*)d\omega_* .\label{eq:integration kernel_1}
\end{equation}

Note that the two kinetic equations \eqref{eq:general kinetic_0}-\eqref{eq:integration kernel_0} and \eqref{eq:general kinetic_1}-\eqref{eq:integration kernel_1} are related: in the first one, the alignment mechanism is nonlocal, using an observation kernel depending on the space variable, and the second form which is purely local, corresponds to the principal order of a scaling limit of the first one by using a rescaling $x' = \eps x$ and changing $\sigma$ and $\nu$ in $\sigma/\eps$ and $\nu/\eps$, where $\eps$ is a small space-time scaling parameter.

The typical example we have in mind for $K$ and $k$ are 
\begin{equation}\label{hyp:K}
	\mathbf{k}(\omega,\omega_*)=\int_{\R^d } \mathbf{K}\pa{\abs{x_*},x_*\cdot \omega, -x_*\cdot \omega_*}\:dx_*=a \omega_* + b\omega_*\otimes \omega_* \:\omega,
\end{equation}
where $a\ge 0$ and $b\ge 0$.

This type of equations has been studied for different kinds of  kernels $\mathbf{K}$ and $\mathbf{k}$. As examples, we mention the Doi-Onsager seminal equation \cite{Ons, Doi}, where $\mathbf{k}(\omega,\omega_*) = \abs{\omega \times  \omega_*}$, 
which is used in the study of rod-like polymer suspensions.  There are also variations of the continuous Vicsek model, where $\mathbf{k}(\omega,\omega_*) = \omega_*$, which lead to the dipolar phase transition \cite{FrouLiu} or to more involved phase transitions \cite{DegFrouLiu1, DegFrouLiu2} with the coefficients $\sigma$ and $\nu$ depending on the average velocity of the solution $\int_{\S} \omega_*f(\omega_*)\: d\omega_*$. We also refer to  \cite{GamKan, FigKanMor} for the study of these specific Kolmogorov-Vicsek type models.

More recently, in \cite{DegFrouMerTre}, a model showing nematic alignement properties has been informally studied. 
Our problem therefore relates to the first moment issue \cite{FrouLiu} and the recent second moments \cite{DegFrouMerTre}.

\par A rigorous Cauchy theory has been proven in the case $\mathbf{k}(\omega,\omega_*) = \omega_*$ in every Sobolev space $H^s$ in the spatially homogeneous framework \cite{FrouLiu}. It appears that the proofs given in \cite{FrouLiu} seem directly applicable to our framework as long as we are working in the spatially homogeneous setting. The non spatially homogeneous setting requires the control of convolution in $x$ and $v$ for the nonlinearity $\mathbf{F_0}[f]$ or a mere integral over $\S^{d-1}$ for $\mathbf{F_1}[f]$. If the first case seems to be controllable by the $L^1_{x,\omega}$-norm of $f$, the second case would need a $L^\infty_xL^1_\omega$ control over the solution. However, it is possible to  link the two equations \eqref{eq:general kinetic_0} and \eqref{eq:general kinetic_1} by using a rescaling $x' = \eps x$ that does not leave the $L^1_x$-norm invariant. We thus choose to establish a Cauchy theory in $L^\infty_{x,\omega}$.

\par Such a theory in non spatially homogeneous framework has been developed in \cite{GamKan} for $\mathbf{k}(\omega,\omega_*) = \omega_*$, with compactness and positivity arguments. However, the nonlinearity $\mathbf{F}$ in \cite{GamKan} is independent of $\omega$ (which brings out a nice surface gradient form) and is renormalised so that it is always bounded by $1$, and this bound disappears in our framework, preventing for efficiently closing $L^\infty$ estimates. However, for the local-in-time Cauchy theory the preliminary study in $L^p_{x,\omega}$ spaces we develop is reminiscent of their approach, even if the $L^\infty$ estimates have to come in hand with another $L^r$ bound. The initial datum will thus be assumed to be non-negative and of finite mass (which is the physically relevant framework). Note that a $L^2$ approach has been recently obtained for more general interactions in \cite{BDM}.

\begin{theo}[Short times Cauchy theories]\label{theo: short time cauchy}
Let $f_0$ be a non-negative function in $L^1_{x,\omega}\cap L^\infty_{x,\omega}$. There exists $T_{\max} >0$ such that there exists a unique $f$ in \\$L^\infty\pa{[0,T_{\max});L^1_{x,\omega}\cap L^\infty_{x,\omega}}$ weak solution to \eqref{eq:general kinetic_0}-\eqref{eq:integration kernel_0} or \eqref{eq:general kinetic_1}-\eqref{eq:integration kernel_1} with initial datum $f_0$. Moreover, $f$ is almost everywhere non-negative.
\end{theo}

Note that for the latter local-in-time theory, the non-negativity and $L^1_{x,\omega}$ requirements are only needed to get $L^\infty_{x,\omega}$ estimates \textit{via} a limit procedure. If one derives direct $L^\infty_{x,\omega}$ estimates then one can drop these assumptions in the short time existence of solutions. We also establish continuous dependency with respect to initial data in every $L^p_{x,\omega}$ for $p<+\infty$. Here again, the non-boundedness of the operators $\mathbf{F_0}$ and $\mathbf{F_1}$ prevents us to reach $p=+\infty$.


Furthermore, preservation of mass and positivity play a crucial role and we are able to extend to global-in-time results in the case \eqref{eq:general kinetic_0}. 


\begin{prop}\label{prop:Global nonlinear}
	Let $u_0 \geq 0$ be in $L^1_{x,\omega}\cap L^\infty_{x,\omega}$. There exists a unique $u$ in $L^\infty\pa{[0,+\infty);L^1_{x,\omega}\cap L^\infty_{x,\omega}}$ weak solution to \eqref{eq:general kinetic_0}-\eqref{eq:integration kernel_0} 
	with initial datum $u_0$. Moreover $u$ is non-negative and for all $t\geq 0$, it satisfies
	\[
	\int_{\R^d\times \S^{d-1}} u(t,x,\omega)dxd\omega = \int_{\R^d\times \S^{d-1}} u_0(x,\omega)dxd\omega :=M_0,
	\]
	and
	\[
	\norm{u(t)}_{L^\infty_{x,\omega}} \leq \norm{u_0}_{L^\infty_{x,\omega}} e^{C_\infty M_1 t},
	\]
	with $C_\infty>0$ depending on $\sigma$, $\nu$ and $\norm{\mathbf{K}}_{L^\infty_{x,,x_*,\omega_*}W^{1,\infty}_\omega}$.
\end{prop}
A control over $\mathbf{F_1}$ cannot be derived only \textit{via} the mass of the initial data and a global in time existence cannot be obtained with our methods. However, using the link between the two kinetic equations \eqref{eq:general kinetic_0}-\eqref{eq:integration kernel_0} and \eqref{eq:general kinetic_1}-\eqref{eq:integration kernel_1}, we derive an approximate global Cauchy theory out of a rescaling of \eqref{eq:general kinetic_0}. Of important note is the fact that the approximate Cauchy theory is derived thanks to the specific spatial convolution form of $\mathbf{F_0}$, see \eqref{eq:integration kernel_0}, whilst all the other results are obtained for $\mathbf{K}= \mathbf{K}(x,x_*,\omega,\omega_*)$.

\begin{prop}\label{prop:approximate r=1}
Let $v_0 \geq 0$ be in $L^1_{x,\omega}\times L^\infty_{x,\omega}$ and let $0<\eps<1$. Define $u_0(x,\omega) = v_0(\eps x ,\omega)$ and let $u$ be the solution of \eqref{eq:general kinetic_0}-\eqref{eq:integration kernel_0} constructed in Proposition \ref{prop:Global nonlinear} associated to $u_0$.
	Then $u^{\eps}(t,x,\omega) = u\pa{\frac{t}{\eps},\frac{x}{\eps},\omega}$ is a global weak approximate solution with initial datum $v_0$ in the sense
	$$\partial_t u^{\eps} + \nabla_x\cdot\pa{c\omega u^{\eps}} = \sigma \Delta_\omega u^{\eps} +\nu \nabla_\omega \cdot \pa{u^\eps \mathbf{P}_{\omega^\bot}(\mathbf{F_1}[u^{\eps}])} + R_\eps(t,x,v)$$
	with the remainder $R_\eps$ such that $\forall \phi  \in C^\infty_c([0,+\infty)\times\R^d\times\S^{d-1})$,
	$$
	\abs{\int_{0}^t\int_{\R^d\times\S^{d-1}} R_\eps(s,x,\omega)\phi(s,x,\omega)dsdxd\omega} \leq \eps^2 \norm{\nabla_\omega\phi}.$$
\end{prop}

From this, we can deduce the global approximate Cauchy theory for \eqref{eq:general kinetic_1}-\eqref{eq:integration kernel_1} in the case where $\sigma= \eps \sigma_0$ and  $\nu= \eps \nu_0$, with $\sigma_0>0$ and $\nu_0>0$.


Each of the theorems above will be tackled separately. We start by a justification of the models in Section \ref{sec:bio} and a general study of linear Fokker-Planck equations in $L^p$ spaces in Section \ref{sec:Cauchy linear}. Then, in Section \ref{sec:Cauchy nonlinear}, we address short-time existence and uniqueness for both 
equations. At last, Sections \ref{sec:Cauchy global} and \ref{sec:approx} takes care of the issue of global results.

\section{The setting}

In this part we justify the equations \eqref{eq:general kinetic_0} and \eqref{eq:general kinetic_1} we study here and we recall some formulas for the computation on the sphere.

\subsection{Justification of the kinetic equations }\label{sec:bio}

The past fifteen years have seen an upsurge of studies, especially in physics, on active particles that have resulted in the apparition of the so-called field of active matter. 

Active (or self-propelled) particle systems have the ability to explain the emergence of collective phenomena that occur, for example, in the collective behavior of animals, flocks of birds \cite{Hildenbrandt}, schools of fish \cite{Hemelrijk}, pedestrian dynamics \cite{Maury}, and microswimmers \cite{David}. 

The simplest interactions between self-propelled particles can have striking consequences at the collective level. For example, a simple pairwise alignment interaction between self-propelled particles, as introduced by Vicsek et al. \cite{Vicsek, Vicsek2}, can lead to large-scale collective clustering and motion in settings where long-range ordering and phase separation would be forbidden for systems in equilibrium \cite{Toner2005HydrodynamicsAP, Marchetti}.



Let us now give more details about Vicsek's model.
Consider $N$ particles: particle $k$ has position $x_k \in \R^d$ and velocity $v_k \in \R^d$, $d\ge 2$. They satisfy the following stochastic differential equations:
\begin{align}
	&dx_k = c\, \omega_k\, dt , \label{eq:syst_part_1}\\ 
	&d \omega_k = \mathbf{P}_{\omega_k^\perp} \circ \big(\nu\, (\bar{\omega}_{k} - \omega_{k})\,   dt + \sqrt{2\sigma}\, dB_t^{k}\big),\label{eq:syst_part_2}
\end{align}
where $\mathbf{P}_{\omega^\perp} = (\text{Id} - \omega\otimes \omega)$ is the orthogonal projection onto the orthogonal plane to $\omega$. 

Let us briefly explain the meaning of each term in this system of equations. Equation \eqref{eq:syst_part_1} describes the transport of the particles: particle $k$ moves in the orientation $\omega_k$ at speed $c > 0$. Equation \eqref{eq:syst_part_2} gives the evolution of the orientations over time. It includes two competing forces. On one hand the first term represents organized motion: particles try to adopt the same orientation. The constant $\nu  > 0$ gives the relaxation speed. On the other hand $\left( B_t^{k}\right)_k$ are N independent Brownian motions and they introduce noise in the dynamics, driving particles away from organized motion.  The symbol ``$\circ$" is used to specify that \eqref{eq:syst_part_2} has to be understood in the Stratonovich sense. 

The k-th particle tends to align itself with the direction $\bar{\omega}_{k} \in \R^{d}$ defined by:
\begin{align}\label{eq:interaction_0}
	\bar{\omega}_k &= \alpha \sum_{\begin{subarray}{l} j,  | x_{j} - x_{k}| < R 
	\end{subarray}} \omega_{j} ,
\end{align}
The  k-th particle thus aligns its velocity with neighboring cells.



Let us now give informally the mean-field limit as the number of agents $N\to  \infty$ of the system \eqref{eq:syst_part_1} - \eqref{eq:syst_part_2}. To do so, we construct the empirical particle measure $f^N (t)$, given by
\[
f^N (t)(x, \omega) = f^N ( x,\omega,t) = \frac1N \sum_{i=1}^N \delta_{(X_i(t),\omega_i(t))}(x, v),
\]
such that $f^N(0) = f_0^N$. Then $f^N(t)$ is a random measure on $\R^{ d} \times \S^1$ for all $t \ge  0$. For certain types of particle systems, see \cite{2011_BolleyCanCar_Vicsek} e.g., it can be shown that, when $N\to \infty$, $f^N$ converges to a deterministic function $f(x, \omega, t)$ with $x\in \R^{d} , \omega \in \S^1$ and $t > 0$ that satisfies a partial differential equation. In the case of \eqref{eq:interaction_0}, the same result is conjectured: the function $f(x, \omega, t)$ satisfies the following mean field kinetic equation:
\begin{equation}\label{eq:cin_0}
	\partial_t f + \nabla_{x} \cdot (c\, \omega f) = \sigma  \, \Delta_\omega f+ \nu  \nabla_\omega\cdot\Big(P_{\omega^{\perp}}\mathbf{F}[f] f\Big)  ,
\end{equation}
where $F$ denotes the  interaction force:
\begin{equation*}
	\mathbf{F}[f](x,\omega,t) = \int_{y\in B(x,R),w\in\S^{1}}   w\, f(y,w,t)\, dydw.
\end{equation*}

Here, we consider the situation where the interaction can be described by a more general kernel than \eqref{eq:interaction_0}.

\subsection{Useful formulas for the computation on the sphere}

We recall here some formulas on the sphere $\S^{d-1}$, which will be used in the sequel.

Let 
$f,g : \S^{d-1} \to  \R$ be scalar-valued function. The following formula for integration by parts hold true, see \cite{Otto_2008} for details:
\begin{equation}\label{eq:IPP_3}
	\int_{\S^{d-1}} f \nabla_\omega g \:d\omega = - \int_{\S^{d-1}} g \nabla_\omega f \:d\omega + (d-1) \int_{\S^{d-1}}\omega fg  \:d\omega .
\end{equation}	

Moreover, since $\nabla_{\omega} f$ is a tangent vector-field, from the definition of the projection
$\mathbf{P}_{\omega^{\perp}}$, we deduce that
\begin{equation}\label{eq:IPP_4}
	\begin{cases}
		\mathbf{P}_{\omega^{\perp}} \omega  &=0\\
		\mathbf{P}_{\omega^{\perp}} \nabla_\omega f &= \nabla_\omega f. 
	\end{cases}
\end{equation}

Furthermore, for any constant vector $v \in \R^d$ we have
\begin{equation}\label{eq:IPP_5}
	\begin{cases}
		 \nabla_\omega(\omega \cdot v ) &= \mathbf{P}_{\omega^{\perp}} v\\
		\nabla_\omega \cdot  \left( \mathbf{P}_{\omega^{\perp}} v\right)  &= -(d-1) \omega \cdot v. 
	\end{cases}
\end{equation}


\section{Cauchy theory for a linear Fokker-Planck equation}\label{sec:Cauchy linear}

In order to deal with the non-linearity $\mathbf{F_0}[f]$ or $\mathbf{F_1}[f]$ we first establish a $L^\infty_{x,\omega}$ Cauchy theory to an associated linear problem.
To that purpose we give the following existence result. It is standard in evolution parabolic equation (see for instance \cite[Section 7.1]{Evans}) and we refer specifically here to a result by Degond \cite[Appendix A]{Deg}.

\begin{theo}\label{theo:cauchy linear}
Let $u_0$ be in $L^2_{x,\omega}$ and $\mathbf{F}$ in $L^\infty_{t,x}W^{1,\infty}_{\omega}$. There 
exists a unique $u$ in the space
\[
Y = \left\{u \in L^2_{t,x}H^1_\omega, \quad \partial_t u + \omega \cdot \nabla_x u \in L^2_{t,x}H^{-1}_\omega\right\},
\]
to the linear diffential equation, for all $(t,x,\omega) \in \R^+\times \R^d  \times \S^{d-1}$:
\begin{equation}\label{eq:linear}
\partial_t u + \nabla_x\cdot\pa{c\omega u} = \sigma  \Delta_\omega u +\nu \nabla_\omega \cdot \pa{u \mathbf{P}_{\omega^\bot}\mathbf{F}}.
\end{equation}
Moreover, if $u_0$ is non-negative then $u$ is non-negative.
\end{theo}

Note that in \cite[Appendix A]{Deg}, the velocity space is $\R^d$ and the Fokker-Planck equation is not written with zeroth order terms. However, their proof relies on an abstract result by Lions \cite{Lio} and on a modification of the equation that includes a zeroth order term, that are directly applicable in our context. Indeed the only difference would come from the integration by parts with the tangential derivative $\nabla_\omega$ that generates additional terms, see \eqref{eq:IPP_3},
 in the weak formulation. However, the latter terms vanish as we are looking at the orthogonal part of $\mathbf{F}$ compared to $\omega$.
\par As we discussed above, we would like a $L^\infty_{x,\omega}$ version of the theorem above. We shall thus prove the following result that is a generalized version of \cite[Lemma 3.1]{GamKan} (note that the $L^1$-norm is not accessible). The Lemma \ref{lem:Lp linear} could be indeed adapted for any $p>1$ (up to changing the constants), as the choice of 3/2 is arbitrary (between 1 and 2).

\begin{lemm}\label{lem:Lp linear}
Let $2\leq p < \infty$, $u_0$ be in $L^p_{x,\omega}$ and $\mathbf{F}$ in $L^\infty_{t,x}W^{1,\infty}_{\omega}$. The solution $u$ to the linear equation \eqref{eq:linear} constructed in Theorem \ref{theo:cauchy linear} is in $L^\infty_t L^p_{x,\omega}$ and gains regularity in $\omega$ since it satisfies
\[
\forall t \geq 0,\quad \norm{u(t)}_{L^p_{x,\omega}} \leq  \norm{u_0}_{L^p_{x,\omega}} \exp\left(C_\infty(\mathbf{F})t\right),
\]
and 
\[
\int_0^t\norm{u^{\frac{p-2}{2}}\nabla_\omega u(s)}^2_{L^2_{x,\omega}}ds \leq \frac{1}{p-\frac{3}{2}} \norm{u_0}^p_{L^p_{x,\omega}} \exp\left(p\, C_\infty(\mathbf{F})t\right),
\]
where $C_\infty(\mathbf{F}) = \nu\norm{\mathbf{F}}_{L^\infty_{t,x}W^{1,\infty}_v} \pa{1 + \frac{\nu}{\sigma} \norm{\mathbf{F}}_{L^\infty_{t,x,\omega}}}$.
\end{lemm}

\begin{proof}[Proof of Lemma \ref{lem:Lp linear}]
Let $u_0$ be in $L^p_{x,\omega}$ for $p\geq 2$, thus $u_0$ belongs $L^2_{x,\omega}$ so we can use Theorem \ref{theo:cauchy linear} and construct $u$ solution to \eqref{eq:linear}. We compute by integrating by parts
\begin{equation}\label{eq:FKPL_1}
\begin{split}
\frac{1}{p}\frac{d}{dt}\norm{u(t)}^p_{L^p_{x,\omega}} =& \int_{\R^d  \times \S^{d-1}} \mbox{sgn}(u(t,x,v))\abs{u(t,x,v)}^{p-1}\partial_t u(t,x,\omega)\:dxd\omega
\\=& -  \int_{\R^d \times\S^{d-1}}\mbox{sgn}(u)\abs{u}^{p-1}\nabla_x(c\omega u)\:dxd\omega 
\\&\quad-\sigma \int_{\R^d \times \S^{d-1}} \nabla_\omega \pa{\mbox{sgn}(u)\abs{u}^{p-1}}\cdot\nabla_\omega u\:dxd\omega 
\\&\qquad+\nu\int_{\R^d \times \S^{d-1}} \mbox{sgn}(u)\abs{u}^{p-1} \nabla_\omega\cdot \pa{u \mathbf{P}_{\omega^\bot}\mathbf{F}}\:dxd\omega.
\end{split}
\end{equation}
The integration by parts does not generate additional terms since the integrand, $\nabla_\omega$, 
is orthogonal to $\omega$.
\par Since $\mbox{sgn}(u)\abs{u}^{p-1}\nabla_x(c\omega u) = \nabla_x(c\omega \abs{u}^p)$, the first term on the right-hand side in \eqref{eq:FKPL_1} is zero. Also, we see that
\[
\nabla_\omega \pa{\mbox{sgn}(u)\abs{u}^{p-1}}\cdot\nabla_\omega u = (p-1)\abs{u}^{p-2} \pa{\nabla_\omega u}^2, 
\]
and
\[
\abs{\abs{u}^{p-1} \nabla_\omega\cdot \pa{u \mathbf{P}_{\omega^\bot}\mathbf{F}}} \leq  \abs{u}^p\norm{\mathbf{F}}_{L^\infty_{t,x}W^{1,\infty}_\omega} + \abs{u}^{p-1}\abs{\nabla_\omega u}\norm{\mathbf{F}}_{L^\infty_{t,x,\omega}}.
\]
Using Cauchy-Schwarz on the third term on the right-hand side in \eqref{eq:FKPL_1}, we end up with
\begin{equation}\label{eq:final ineq}
\begin{split}
\frac{1}{p}\frac{d}{dt}\norm{u(t)}^p_{L^p_{x,\omega}} &\leq -(p-1)\sigma\norm{u^{\frac{p-2}{2}}\nabla_\omega u}_{L^2_{x,\omega}}^2+ \nu \norm{\mathbf{F}}_{L^\infty_{t,x}W^{1,\infty}_\omega}\norm{u(t)}^p_{L^p_{x,\omega}}\\
&\qquad  + \nu\norm{\mathbf{F}}_{L^\infty_{t,x,\omega}}\norm{u}_{L^p_{x,\omega}}^{\frac{p}{2}}\norm{u^{\frac{p-2}{2}}\nabla_\omega u}_{L^2_{x,\omega}}
\\&\leq - \pa{(p-1)\sigma - \frac{\eta \nu\norm{\mathbf{F}}_{L^\infty_{t,x,v}}}{2}}\norm{u^{\frac{p-2}{2}}\nabla_\omega u}^2_{L^2_{x,\omega}} 
\\&\qquad+ \nu\norm{\mathbf{F}}_{L^\infty_{t,x}W^{1,\infty}_\omega}\pa{1+\eta^{-1}}\norm{u(t)}^p_{L^p_{x,\omega}},
\end{split}
\end{equation}
for any $\eta >0$.
\par Choosing $\eta$ sufficiently small such that $\frac{\sigma}{2} - \frac{\eta}{2}\nu\norm{\mathbf{F}}_{L^\infty_{t,x,v}} \geq 0$ we can apply Gr\"onwall lemma and for all $t \geq 0$ get
\begin{equation*}
\begin{split}
\norm{u(t)}^p_{L^p_{x,\omega}}+ & \left(p-\frac{3}{2}\right)\sigma\int_0^t  \norm{u^{\frac{p-2}{2}}(s)\nabla_\omega u(s)}^2_{L^2_{x,\omega}}  
\\
& \qquad \qquad \qquad \times 
 \exp\left(p\nu\norm{\mathbf{F}}_{L^\infty_{t,x}W^{1,\infty}_v}(1+\eta^{-1})(t-s)\right)
ds \\
&\leq \norm{u_0}_{L^p_{x,\omega}}^p \exp\left( p\nu\norm{\mathbf{F}}_{L^\infty_{t,x}W^{1,\infty}_v}(1+\eta^{-1})t\right).
\end{split}
\end{equation*}
The latter is exactly Proposition \ref{prop:Linfty linear} for $p$ in $[2,+\infty)$.
\end{proof}

We would like to take the limit $p\to+\infty$ in the above Lemma \ref{lem:Lp linear} when starting from an initial data in $L^\infty_{x,\omega}$. This is straightforward in the case of a bounded spatial domain, but, in the unbounded case $\R^d$, it requires that the function $u$ belongs to the space $L^s_{x,\omega}$ for a given $s$. This is the purpose of the following proposition. We emphasize again the importance of the positivity of the initial datum.

\begin{prop}\label{prop:Linfty linear}
Let $u_0$ be in $L^1_{x,\omega} \cap L^\infty_{x,\omega}$ and $\mathbf{F}$ in $L^\infty_{t,x}W^{1,\infty}_{\omega}$. If $u_0 \geq 0$, then the solution $u$ to the linear equation \eqref{eq:linear} constructed in Theorem \ref{theo:cauchy linear} satisfies, for all $t \geq 0$
\[
\norm{u(t)}_{L^1_{x,\omega}} = \norm{u_0}_{L^1_{x,\omega}},
\]
and
\[
\norm{u(t)}_{L^\infty_{x,\omega}} \le \norm{u_0}_{L^\infty_{x,\omega}} e^{C_\infty(\mathbf{F})t},
\]
where $C_\infty(\mathbf{F})$ was defined in Lemma \ref{lem:Lp linear}.
\end{prop}

\begin{proof}[Proof of Proposition \ref{prop:Linfty linear}]
Thanks to Theorem \ref{theo:cauchy linear} the non-negativity of $u_0$ implies that $u$ is non-negative. Hence,
\begin{eqnarray*}
\frac{d}{dt}\norm{u(t)}_{L^1_{x,\omega}} &=& \frac{d}{dt}\int_{\R^d\times \S^{d-1}} u(t,x,\omega)\:dxd\omega 
\\&=& \int_{\R^d\times\S^{d-1}}\nabla_x\cdot(c\omega u) + \nabla_\omega\cdot\pa{\sigma\nabla_\omega u - \nu u \mathbf{P_{\omega^\bot}}\mathbf{F}}dxd\omega \\&=& 0
\end{eqnarray*}
because, again, the integration by parts do not generate additional terms since the integrands, $\nabla_\omega$ and $\mathbf{P_{\omega^\bot}}(\mathbf{F})$, are orthogonal to $\omega$. This concludes the preservation of the $L^1_{x,\omega}$-norm.
\par The $L^\infty_{x,\omega}$ estimate follows directly from Lemma \ref{lem:Lp linear} by taking the limit when $p$ goes to $+\infty$ since $u(t)$ also belongs to $L^1_{x,\omega}$ for all $t\geq 0$.
\end{proof}


\section{
Proof of Theorem \ref{theo: short time cauchy}}\label{sec:Cauchy nonlinear}

This section is devoted to the proof of existence and uniqueness for short times for the nonlinear problems \eqref{eq:general kinetic_0}-\eqref{eq:integration kernel_0} and \eqref{eq:general kinetic_1}-\eqref{eq:integration kernel_1}. We begin with the existence issue. 


The proof will follow an iterative scheme. We choose $u^{(0)}(t,x,\omega) = u_0(x,\omega)$ and construct the following iterative scheme where $u^{(n+1)}$ is the solution to
\begin{equation}\label{eq:iteration}
\partial_t u^{(n+1)} + \nabla_x\cdot\pa{c\omega u^{(n+1)}} = \sigma  \Delta_\omega u^{(n+1)} + \nu \nabla_\omega \cdot \pa{u^{(n+1)} \mathbf{P}_{\omega^\bot}\pa{\mathbf{F_i}[u^{(n)}]}},
\end{equation}
on $\R^+\times \R^d  \times \S^{d-1}$, with initial datum $u_0$ and for $i\in \{0,1\}$. By definition of $\mathbf{F_i}[u^{(n)}]$,  for all $t \leq T$, we see that
\begin{equation}\label{eq:norm Flg}
\abs{\mathbf{F_i}[u^{(n)}]}([0,T],x,\omega) + \abs{\nabla_\omega \mathbf{F_i}[u^{(n)}]}([0,T],x,\omega) \leq K_{\infty,i} \left\|u^{(n)}\right\|_{L^\infty_{[0,T],x,\omega}}, 
\end{equation}
with 
\[
K_{\infty,0} = \norm{K}_{L^1_{x}W^{1,\infty}_\omega L^1_{\omega_*}}
\]
and 
\[
K_{\infty,1}=\norm{k}_{W^{1,\infty}_\omega L^1_{\omega_*}}. 
\]
We conclude that $\mathbf{F_i}[u^{(n)}]$ belongs to $L^\infty_{[0,T],x}W^{1,\infty}_\omega$ whenever $u^{(n)}$ belongs to $L^\infty_{[0,T],x,\omega}$.

\bigskip
\noindent\textbf{Step 1: Well-posedness of $\pa{u^{(n)}}_{n\in\N}$  and boundedness for short times.}
Suppose that $u^{(n)}$ is non-negative and belongs to the space $L^1_{x,\omega}\cap L^\infty_{t,x,\omega}$ (which is satisfied by $u^{(0)}$) then thanks to \eqref{eq:norm Flg}, Theorem \ref{theo:cauchy linear} and Proposition \ref{prop:Linfty linear}, $u^{(n+1)}$ is  well-defined, non-negative and belongs to $L^1_{x,\omega}\cap L^\infty_{t,x,\omega}$ and for all $0\leq t \leq T$, it satisfies
\[
\norm{u^{(n+1)}(t)} \leq \norm{u_0}_{L^\infty_{x,\omega}} \exp\left(C_{\infty,i}\norm{u^{(n)}}_{L^\infty_{[0,T],x,\omega}}T\right),
\]
with 
\begin{equation}\label{def:Cinfty}
C_{\infty,i} = \nu K_{\infty,i}(1+\frac{\nu}{\sigma}K_{\infty,i}).
\end{equation}

Therefore, for any $R > 1$, if we set $T_{R,i} = \frac{\mbox{ln}(R)}{RC_{\infty,i}\norm{u_0}_{L^\infty_{x,\omega}}}$, we obtain that
\[
\pa{\norm{u^{(n)}}_{L^\infty_{[0,T_{R,i}],x,\omega}} \leq R\norm{u_0}_{L^\infty_{x,\omega}}} \quad  \Longrightarrow \quad  \pa{\norm{u^{(n+1)}}_{L^\infty_{[0,T_{R,i}],x,\omega}} \leq R\norm{u_0}_{L^\infty_{x,\omega}}},
\]
hence $\pa{u^{(n)}}_{n\in\N}$ is bounded in $L^\infty\pa{[0,T_{R,i}],L^\infty_{x,\omega}}$.

\bigskip
\noindent\textbf{Step 2: Weak-* compactness of $\mathbf{\pa{u^{(n)}}_{n\in\N}}$ and limit equation for short times.} First of all the first step taught us that the sequence $\pa{u^{(n)}}_{n\in\N}$ is bounded in $L^\infty_{[0,T_{R,i}],x,\omega}$ but also that for almost every $(t,x)$ in $[0,T_{R,i}]\times\R^d$, the sequence $\pa{u^{(n)}(t,x)}_{n\in\N}$ is bounded in $L^\infty_\omega$. Therefore, we can first extract a subsequence, still denoted $\pa{u^{(n)}(t,x)}_{n\in\N}$, such that for almost every $(t,x)$, the sequence $\pa{u^{(n)}(t,x)}_{n\in\N}$ converges weakly-* towards $u(t,x,\cdot)$ in $L^\infty_\omega$. Then, for almost every $t$ in $[0,T_{R,i}]$ this subsequence is still bounded in $L^\infty_{x,\omega}$ and as such converges, up to a subsequence that we still denote $\pa{u^{(n)}}_{n\in\N}$, weakly-* in $L^\infty_{x,\omega}$ and from the first almost everywhere convergence its weak-* limit is the function $u(t,\cdot,\cdot)$. Finally, the sequence $\pa{u^{(n)}(t,x)}_{n\in\N}$ being bounded in $L^\infty_{[0,T_{R,i}],x,\omega}$ we can extract a subsequence, still denoted $\pa{u^{(n)}(t,x)}_{n\in\N}$, that converges weakly-* in $L^\infty_{[0,T_{R,i}],x,\omega}$ towards $u$. As a recap: there exists $u \in L^\infty_{[0,T_{R,i}],x,\omega}$
\begin{equation}\label{eq:weak* CV}
\begin{split}
u_n(t,x) \overset{\ast}{\rightharpoonup} u(t,x), &\mbox{ a.e. in}\quad L^\infty_{\omega},
\\
u_n(t) \overset{\ast}{\rightharpoonup} u(t) &\mbox{ a.e. in}\quad L^\infty_{x,\omega},
\\u_n \overset{\ast}{\rightharpoonup} u &\mbox{ a.e. in}\quad L^\infty_{[0,T_{R,i}],x,\omega}.
\end{split}
\end{equation}
We look at the weak formulation of \eqref{eq:iteration}. Multiplying \eqref{eq:iteration} by any function $\phi$ in $C^\infty_c([0,T_{R,i})\times\R^d\times\S^{d-1})$  and integrating by parts, it yields
\begin{equation*}
\begin{split}
&-\int_{\R^d\times\S^{d-1}} u_0(x,\omega)\phi(0,x,\omega)\:dxd\omega \\
&= \int_0^t\int_{\R^d\times\S^{d-1}} u^{(n+1)}\cro{\partial_t\phi + c\omega\cdot\nabla_x\phi-\sigma\Delta_\omega\phi}\:dsdxd\omega 
\\& \qquad + \nu\int_0^t\int_{\R^d\times\S^{d-1}}u^{(n+1)}\mathbf{P_{\omega^\bot}}\pa{\mathbf{F_i}[u^{(n)}]}\cdot\nabla_\omega \phi \:dsdxd\omega.
\end{split}
\end{equation*}
Thanks to the weak-* convergence of $\pa{u^{(n)}}_{n\in\N}$ in $L^\infty_{[0,T_{R,i}],x,\omega}$, the first integral on the right-hand side tends to 
$$\int_0^t\int_{\R^d\times\S^{d-1}} u\cro{\partial_t\phi + c\omega\cdot\nabla_x\phi-\sigma \Delta_\omega\phi}\:dsdxd\omega.$$
It remains to deal with the non-linearity. However, owing to the definitions \eqref{eq:integration kernel_0} and \eqref{eq:integration kernel_1}, we have
\[
\mathbf{P_{\omega^\bot}}\pa{\mathbf{F_0}[u^{(n)}]}(s,x,\omega) = \int_{\R^d\times\S^{d-1}}\mathbf{P_{\omega^\bot}}\pa{\mathbf{K}}(x,\omega,x_*,\omega_*)u^{(n)}(s,x_*,\omega_*)\:dx_*d\omega_* ,
\]
and
\[
\mathbf{P_{\omega^\bot}}\pa{\mathbf{F_1}[u^{(n)}]}(s,x,\omega) = \int_{\S^{d-1}}\mathbf{P_{\omega^\bot}}\pa{\mathbf{k}}(\omega,\omega_*)u^{(n)}(s,x,\omega_*)\:dx_*d\omega_*, 
\]
to the almost everywhere weak-* convergences of $\pa{u^{(n)}}_{n\in\N}$ in $L^\infty_{x,\omega}$ (for the case $i=0$) and in $L^{\infty}_\omega$ (for the case $i=1$), it follows that in every case $\mathbf{P_{\omega^\bot}}\pa{\mathbf{F_i}[u^{(n)}]}$ converges almost everywhere towards $\mathbf{P_{\omega^\bot}}\pa{\mathbf{F_i}[u]}(s,x,\omega)$.  Thanks to Egorov's theorem on sets with finite measure we have
\begin{equation*}
\begin{split}
&\forall r >0,\:\forall \eps >0,\: \exists X_{r,\eps} \subset [0,T_{R,r}]\times B_x(0,r) \times \S^{d-1},\:\mbox{such that}
\\&\abs{\left([0,T_{R,r}]\times B_x(0,r) \times \S^{d-1}\right)\backslash X_{r,\eps}} \leq \eps,\\
 &\mbox{and}\quad \mathbf{P_{\omega^\bot}}\pa{\mathbf{F_r}[u^{(n)}]} \to \mathbf{P_{\omega^\bot}}\pa{\mathbf{F_r}[u]}(s,x,\omega) \:\mbox{uniformly in $X_{r,\eps}$}.
\end{split}
\end{equation*}
Since $\phi$ belongs to $C^\infty_c\left([0,T_{R,i})\times\R^d\times\S^{d-1}\right)$ we can choose $r$ such that $\phi=0$ outside $[0,T_{R,i}]\times B(0,r)\times \S^{d-1}$. Thus, for $n$ large enough by the weak-* convergence of $u^{(n+1)}$, we see that
\begin{equation*}
\begin{split}
&\Bigg|\int_0^t\int_{\R^d\times\S^{d-1}}u^{(n+1)}\mathbf{P_{\omega^\bot}}\pa{\mathbf{F_i}[u^{(n)}]}\cdot\nabla_\omega \phi \:dsdxd\omega \\
& \qquad \qquad - \int_0^t\int_{\R^d\times\S^{d-1}}u\mathbf{P_{\omega^\bot}}\pa{\mathbf{F_i}[u]}\cdot\nabla_\omega \phi \:dsdxd\omega\Bigg|
\\&\leq \abs{\int_0^t\int_{B(0,r)\times\S^{d-1}}\pa{u^{(n+1)} -u)}\mathbf{P_{\omega^\bot}}\pa{\mathbf{F_i}[u]}\cdot\nabla_\omega \phi \:dsdxd\omega} 
\\&\qquad \qquad + \abs{\int_0^t\int_{B(0,r)\times\S^{d-1}}u^{(n+1)}\pa{\mathbf{P_{\omega^\bot}}\pa{\mathbf{F_i}[u^{(n)}]}-\mathbf{P_{\omega^\bot}}\pa{\mathbf{F_i}[u]}}\cdot\nabla_\omega \phi \:dsdxd\omega}
\\&\leq \eps + R\norm{u_0}_{L^\infty_{x,\omega}}\norm{\nabla_\omega\phi}_{L^\infty_{[0,T_{R,i}],x,\omega}}\\
& \qquad \qquad  \times \int_{[0,T_{R,i}]\times B(0,r)\times\S^{d-1}}\abs{\mathbf{P_{\omega^\bot}}\pa{\mathbf{F_i}[u^{(n)}]}-\mathbf{P_{\omega^\bot}}\pa{\mathbf{F_i}[u]}}\:dsdxd\omega
\\&\leq \eps+ R\norm{u_0}_{L^\infty_{x,\omega}}\norm{\nabla_\omega\phi}_{L^\infty_{[0,T_{R,i}],x,\omega}}\abs{X_{r,\eps}}\norm{\mathbf{P_{\omega^\bot}}\pa{\mathbf{F_i}[u^{(n)}]}-\mathbf{P_{\omega^\bot}}\pa{\mathbf{F_i}[u]}}_{L^\infty_{X_{r,\eps}}}
\\&\quad\quad+R\norm{u_0}_{L^\infty_{x,\omega}}\norm{\nabla_\omega\phi}_{L^\infty_{[0,T_{R,i}],x,\omega}} \abs{B(0,r)\backslash X_{r,\eps}}2RK_{\infty,i}\norm{u_0}_{L^\infty_{x,\omega}},
\end{split}
\end{equation*}
where we have used that $\norm{u^{(n)}}_{L^\infty_{[0,T_{R,i}],x,\omega}}\leq R\norm{u_0}_{L^\infty_{x,\omega}}$ and the inequality \eqref{eq:norm Flg} to bound the $\mathbf{F_i}$ term. The uniform convergence on $X_{r,\eps}$ shows that for $n$ large enough the above quantity is bounded by
\[
\eps\pa{1 + Rr^d\norm{u_0}_{L^\infty_{x,\omega}}\norm{\nabla_x\phi}_{L^\infty_{[0,T_{R,i}],x,\omega}}+ 2R^2K_{\infty,i}\norm{u_0}^2_{L^\infty_{x,\omega}}\norm{\nabla_x\phi}_{L^\infty_{[0,T_{R,i}],x,\omega}}},
\]
which shows that the nonlinear term converges too. Finally, in the limit $n$ goes to $+\infty$ we obtain
\begin{equation*}
\begin{split}
-\int_{\R^d\times\S^{d-1}} u_0(x,\omega)\phi(0,x,\omega)\:dxd\omega =& \int_0^t\int_{\R^d\times\S^{d-1}} u\cro{\partial_t\phi + c\omega\cdot\nabla_x\phi-\sigma\Delta_\omega\phi}\:dsdxd\omega 
\\&+ \nu\int_0^t\int_{\R^d\times\S^{d-1}}u\mathbf{P_{\omega^\bot}}\pa{\mathbf{F_i}[u]}\cdot\nabla_\omega \phi \:dsdxd\omega,
\end{split}
\end{equation*}
hence $u$ is a weak solution to \eqref{eq:general kinetic_0}-\eqref{eq:integration kernel_0} or \eqref{eq:general kinetic_1}-\eqref{eq:integration kernel_1}.

To establish uniqueness of the solutions described above we show continuity with respect to initial data.

\begin{prop}\label{prop:continuity initial data}
Let $u_{0,1},u_{0,2}$ be two non-negative functions in $L^1_{x,\omega} \cap L^\infty_{x,\omega}$ and $l$ in $\br{0,1}$. Suppose $u_1,u_2$ are solutions to \eqref{eq:general kinetic_0}-\eqref{eq:integration kernel_0} or \eqref{eq:general kinetic_1}-\eqref{eq:integration kernel_1} defined on $L^\infty\pa{[0,T_{\max});L^1_{x,\omega}\cap L^\infty_{x,\omega}}$ associated to the initial datum $u_{0,1}$ and $u_{0,2}$ respectively.
\\Then, for any $T$ in $(0,T_{\max})$ and all $p$ in $[2,+\infty)$, there exists $C(p,T)>0$ such that
\[
\forall t \in [0,T],\quad \norm{u_1(t)-u_2(t)}_{L^p_{x,\omega}} \leq e^{C(p,T)t}\norm{u_{0,1}-u_{0,2}}_{L^\infty_{x,\omega}}.
\]
Note that the constant $C(p,T)$ depends on $\norm{u_1}_{L^\infty_{[0,T],x,\omega}}$, $\norm{u_2}_{L^\infty_{[0,T],x,\omega}}$ and $\nu$. 
\end{prop}

\begin{rema}
Note that uniqueness only comes in the space $L^\infty\pa{[0,T_{\max,i});L^1_{x,\omega}\cap L^\infty_{x,\omega}}$ and does not require $u\geq 0$ nor preservation of mass. Moreover, our method prevents from having continuity in $L^\infty_{x,\omega}$ because of the dependencies of the parameter in $\sup\limits_{[0,T]}\norm{u(t)}_{L^\infty_{x,\omega}}$.
\end{rema}

\begin{proof}[Proof of Proposition \ref{prop:continuity initial data}]
Let us fix $T$ in $(0,T_{\max,i})$ then $u_1$ and $u_2$ belong to $L^\infty_{[0,T],x,\omega}$ and let us denote 
$$M = \max\br{\norm{u_1}_{L^\infty_{[0,T],x,\omega}}; \norm{u_2}_{L^\infty_{[0,T],x,\omega}}}.$$
Moreover, $u_1$ and $u_2$ being in $L^1_{x,\omega}\cap L^\infty_{x,\omega}$ they also belong to $L^2_{x,\omega}$. Hence we compute for all $t$ in $[0,T]$,
\begin{equation*}
\begin{split}
\partial_t\cro{u_1-u_2} + \nabla_x\cdot \pa{c\omega \cro{u_1-u_2}} =& \sigma \Delta_\omega \cro{u_1-u_2} \\
&+  \frac{\nu}{2}\nabla_\omega\cdot\pa{\cro{u_1-u_2}\mathbf{P}_{\omega^\bot}\pa{\mathbf{F_i}[u_1+u_2]}}
\\& +\frac{\nu}{2}\nabla_\omega\cdot \pa{\cro{u_1+u_2}\mathbf{P}_{\omega^\bot}\pa{\mathbf{F_i}[u_1-u_2]}},
\end{split}
\end{equation*}
owing to the algebraic identity : $2(ab-cd) = (a-c)(b+d) + (a+c)(b-d)$. Since $u_1(t)$ and $u_2(t)$ belong to $L^1_{x,\omega} \cap L^\infty_{x,\omega}$ they also belong to $L^p_{x,\omega}$ for any $p\geq 1$.

Performing similar computations as those leading to \eqref{eq:final ineq} in the proof of  Lemma \ref{lem:Lp linear}, we obtain
\begin{equation}\label{eq:uniqueness start}
\begin{split}
\frac{1}{p}&\frac{d}{dt}\norm{u_1-u_2}^p_{L^p_{x,\omega}} 
\\&\leq -\pa{(p-1)\sigma - \eta C(M, \nu)}\norm{(u_1-u_2)^{\frac{p-2}{2}}\nabla_\omega(u_1-u_2)}_{L^2_{x,\omega}}^2\\
& \quad+ C(M, \nu,\eta)\norm{u_1-u_2}^p_{L^p_{x,\omega}} + \frac{\nu}{2}\Big\|\int_{\R^d\times\S^{d-1}}\mbox{sgn}(u_1-u_2)\abs{u_1-u_2}^{p-1} \times
\\&\qquad \qquad \qquad \qquad \qquad \nabla_\omega\cdot\cro{(u_1+u_2)\mathbf{P}_{\omega^\bot}\pa{\mathbf{F_i}[u_1-u_2]}}dxd\omega\Big\|.
\end{split}
\end{equation}
Here and after we denote by $C(a,b)$ any positive constant depending only on the quantities $a$ and $b$. 

We integrate by parts the last term, as usual there are no extra terms thanks to the orthogonality of the integrand, and then we use the Cauchy-Schwarz inequality followed by the H\"older one:
\begin{eqnarray}
&&\abs{\int_{\R^d\times\S^{d-1}}\mbox{sgn}(u_1-u_2)\abs{u_1-u_2}^{p-1} \nabla_\omega\cdot\cro{(u_1+u_2)\mathbf{P}_{\omega^\bot}\pa{\mathbf{F_i}[u_1-u_2]}}dxd\omega} \nonumber
\\&& = (p-1)\abs{\int_{\R^d\times\S^{d-1}}(u_1+u_2)\abs{u_1-u_2}^{p-2}\nabla_\omega\cro{u_1-u_2}\cdot \mathbf{P}_{\omega^\bot}\pa{\mathbf{F_i}[u_1-u_2]}dxd\omega} \nonumber
\\&& \leq 2M(p-1)\norm{(u_1-u_2)^{\frac{p-2}{2}}\nabla_\omega(u_1-u_2)}_{L^2_{x,\omega}} \nonumber\\
&& \qquad \qquad \times \cro{\int_{\R^d\times \S^{d-1}}\abs{u_1-u_2}^{p-2}\pa{\mathbf{P}_{\omega^\bot}\pa{\mathbf{F_i}[u_1-u_2]}}^2dxd\omega}^{\frac{1}{2}} \nonumber
\\&&\leq 2M(p-1)\norm{(u_1-u_2)^{\frac{p-2}{2}}\nabla_\omega(u_1-u_2)}_{L^2_{x,\omega}}\nonumber \\
&& \qquad \qquad  \times  \norm{u_1-u_2}_{L^p_{x,\omega}}^{\frac{p-2}{2}}\norm{\mathbf{P}_{\omega^\bot}\pa{\mathbf{F_i}[u_1-u_2]}}_{L^p_{x,\omega}} . \label{eq:control 2nd term}
\end{eqnarray}
The term $\norm{\mathbf{P}_{\omega^\bot}\pa{\mathbf{F_i}[u_1-u_2]}}_{L^p_{x,\omega}}$ must be controlled differently depending on the value of $i$. Nevertheless, in both cases it is a kernel operator and an H\"older inequality yields that
\[
\norm{\mathbf{P}_{\omega^\bot}\pa{\mathbf{F_i}[u_1-u_2]}}_{L^p_{x,\omega}} \leq K_{i,p}\norm{u_1-u_2}_{L^p_{x,\omega}},
\]
with $K_{0,p}= \norm{\mathbf{K}}_{L^q_{x,x_*,\omega,\omega_*}}$ and $K_{1,p}= \norm{\mathbf{k}}_{L^q_{\omega,\omega_*}}$, where $q$ is the conjugate exponent of $p$. Using it inside \eqref{eq:control 2nd term} and plugging it back into \eqref{eq:uniqueness start} yields
\begin{equation*}
\begin{split}
&\frac{d}{dt}\norm{u_1-u_2}^p_{L^p_{x,\omega}} \\
&\leq -p\pa{(p-1)\sigma - \eta C(M, \nu)}\norm{(u_1-u_2)^{\frac{p-2}{2}}\nabla_\omega(u_1-u_2)}_{L^2_{x,\omega}}^2\\
&\quad+ pC(M, \nu,\eta)\norm{u_1-u_2}^p_{L^p_{x,\omega}} \\
&\quad+ p(p-1)MK_{i,p}\nu \norm{(u_1-u_2)^{\frac{p-2}{2}}\nabla_\omega(u_1-u_2)}_{L^2_{x,\omega}}\norm{u_1-u_2}_{L^p_{x,\omega}}^{\frac{p}{2}}
\\&\leq \pa{\eta C(M, \nu,i,p)-\sigma p(p-1)}\norm{(u_1-u_2)^{\frac{p-2}{2}}\nabla_\omega(u_1-u_2)}_{L^2_{x,\omega}}^2 
\\&\quad+ C(M, \nu,i,p,\eta)\norm{u_1-u_2}^p_{L^p_{x,\omega}}.
\end{split}
\end{equation*}
Choosing $\eta$ sufficiently small, the first term on the right-hand side is negative and we conclude thanks to Gr\"onwall lemma.
\end{proof}


\section{Proof of Proposition \ref{prop:Global nonlinear}}\label{sec:Cauchy global}

\textbf{Step 1: Preservation of mass.} The key ingredient of the present proof is to establish conservation of mass when passing to the limit in the iterative scheme used in the proof of Theorem \ref{theo: short time cauchy}. More precisely, we recall that the sequence $\pa{u^{(n)}}_{n\in\N}$ is defined by
\[
\partial_t u^{(n+1)} + \nabla_x\cdot\pa{c\omega u^{(n+1)}} = \sigma \Delta_\omega u^{(n+1)} + \nu \nabla_\omega \cdot \pa{u^{(n+1)} \mathbf{P}_{\omega^\bot}\pa{\mathbf{F_0}[u^{(n)}]}}.
\]
Using Proposition, \ref{prop:Linfty linear} we know that for all $n$, $u^{(n)}\geq 0$ and that $\norm{u^{(n)}(t)}_{L^1_{x,\omega}} = M_1$ for all $t\geq 0$.
\par Consider $\chi_R(x,\omega)$ a smooth test function such that $0\leq \chi_R\leq 1$ with $\chi_R =1$ on $B(0,R)$, $\chi_R =0$ outside $B(0,R+1)$. We can choose $\chi_R$ such that $\abs{\nabla_{x,\omega} \chi_R} \leq C_d$ and $\abs{\Delta_{x,\omega}\chi_R} \leq C_d$ where $C_d>0$ only depends on the dimension and is independent of $R$. From the constructive Theorem \ref{theo:cauchy linear} we also recall that $u^{(n)}$ belongs to $L^2_{t,x}H^1_\omega$ and that $\partial_t u^{(n)} + \omega\cdot\nabla_x u^{(n)}$ belongs to $L^2_{t,x}H^{-1}_\omega$. We therefore use the weak formulation of the iterative scheme with $\chi_R$ and we get that for any $t\geq 0$,
\begin{equation*}
\begin{split}
&\abs{\int_{\R^d\times\S^{d-1}} u^{(n+1)}(t,x,\omega)\chi_R(x,\omega)dxd\omega - \int_{\R^d\times\S^{d-1}} u_0(x,\omega)\chi_R(x,\omega)dxd\omega} 
\\&\quad\quad= \Bigg|\int_0^t\int_{\R^d\times\S^{d-1}} u^{(n+1)}\bigg\{\partial_t\chi_R + c\omega\cdot\nabla_x\chi_R - \sigma\nabla_\omega\chi_R \\
&\qquad \qquad +\nu \mathbf{P_{\omega^\bot}}\pa{\mathbf{F_0}(u^{(n)})}\cdot\nabla_\omega\chi_R\bigg\}dxd\omega \Bigg|
\\&\quad\quad \leq \pa{cC_d+\sigma C_d + \nu C_d\norm{\mathbf{P_{\omega^\bot}}\pa{\mathbf{F_0}(u^{(n)})}}_{L^\infty_{[0,t],x,\omega}}}\\
& \qquad \qquad \times \int_0^t\int_{B(0,R+1)\backslash B(0,R)}\abs{u^{(n+1)}}dxd\omega.
\end{split}
\end{equation*}
Since the mass is preserved we can evaluate $\norm{\mathbf{P_{\omega^\bot}}\pa{\mathbf{F_0}(u^{(n)})}}_{L^\infty_{[0,t],x,\omega}}$:
\begin{equation}\label{eq:control K M1}
\begin{split}
&\abs{\mathbf{P_{\omega^\bot}}\pa{\mathbf{F_0}[u^{(n)}]}(s,x,\omega)} \\
&\qquad \leq \int_{\R^d\times\S^{d-1}}\abs{\mathbf{P_{\omega^\bot}}\pa{\mathbf{K}}(x,\omega,x_*,\omega_*)}u^{(n)}(s,x_*,\omega_*)\:dx_*d\omega_*
\\&\qquad\leq \norm{\mathbf{K}}_{L^\infty_{x,\omega,x_*,\omega_*}}M_1
\end{split}
\end{equation}
which is time independent. Hence, there exists $C$ independent of time and $R$ such that
\begin{equation*}
\begin{split}
&\abs{\int_{\R^d\times\S^{d-1}} u^{(n+1)}(t,x,\omega)\chi_R(x,\omega)dxd\omega - \int_{\R^d\times\S^{d-1}} u_0(x,\omega)\chi_R(x,\omega)dxd\omega} 
\\&\quad\quad\quad \leq C\int_0^t\int_{B(0,R+1)\backslash B(0,R)}\abs{u^{(n+1)}}dxd\omega.
\end{split}
\end{equation*}
Thanks to the weak-* convergence of $\pa{u^{(n)}}_{n\in\N}$ towards $u$ on a given $[0,T]$ (up to a subsequence, see Step 2 of the proof of Theorem \ref{theo: short time cauchy}) we can take the limit as $n$ goes to $+\infty$ and obtain
$$\abs{\int_{\R^d\times\S^{d-1}} u\:\chi_R\: dxd\omega - \int_{\R^d\times\S^{d-1}} u_0\:\chi_R\:dxd\omega} \leq C\int_0^t\int_{B(0,R+1)\backslash B(0,R)}u \: dxd\omega.$$
To conclude we first invoke the dominated convergence theorem ($u$ and $u_0$ are positive and integrable and $\chi_R$ is bounded by one) to take the limit $R$ goes to $+\infty$ and see that the left-hand side converges to $\abs{\int u - \int u_0}$. For the right-hand side we also invoke the dominated convergence theorem to see that it tends to $0$ as $R$ tends to $+\infty$, because $u(t)$ is integrable. At last we showed that
$$\forall t \in [0,T], \: \int_{\R^d\times\S^{d-1}}u(t,x,\omega)dxd\omega = M_1.$$

\textbf{Step 2: Global solution and estimate.} In Step 1 we built a solution to \eqref{eq:general kinetic_0}-\eqref{eq:integration kernel_0} on $[0,T_{\max,0})$ and thanks to the uniqueness of Proposition \ref{prop:continuity initial data} it follows that the solution associated to $u_0$ on $[0,T]$ preserves the mass. Let us now prove that $T_{\max,0} = +\infty$.
\\Using Proposition \ref{prop:Linfty linear} we have the following implicit control over $u$:
\[
\forall t \in [0,T_{\max,0}),\quad \norm{u(t)}_{L^{\infty}_{x,\omega}} \leq \norm{u_0}_{L^\infty_{x,\omega}} e^{C_\infty\pa{\mathbf{P_{\omega^\bot}}(\mathbf{F_0(u)})}t},
\]
where
\begin{eqnarray*}
C_\infty\pa{\mathbf{P_{\omega^\bot}}(\mathbf{F_0(u)})} &=& \nu \norm{\mathbf{P_{\omega^\bot}}(\mathbf{F_0(u)})}_{L^\infty_{t,x}W^{1,\infty}_\omega}\pa{1+\frac{\nu}{\sigma}\norm{\mathbf{P_{\omega^\bot}}(\mathbf{F_0(u)})}_{L^\infty_{t,x,\omega}}}
\\&\leq& \nu \norm{\mathbf{K}}_{L^\infty_{x,x_*,\omega_*}W^{1,\infty}_\omega}M_1\pa{1+\frac{\nu}{\sigma}\norm{\mathbf{K}}_{L^\infty_{x,x_*,\omega_*,\omega}}M_1}.
\end{eqnarray*}
The estimate directly comes from \eqref{eq:control K M1} applied to $\mathbf{K}$ and $\nabla_\omega\mathbf{K}$. We thus see that we have a constant $C_\infty >0$ only depending on the parameter of the subject : $\nu$, $\sigma$, $\mathbf{K}$ and $M_1$ such that
\[
\forall t \in [0,T_{\max,0}),\quad \norm{u(t)}_{L^{\infty}_{x,\omega}} \leq \norm{u_0}_{L^\infty_{x,\omega}} e^{C_\infty t}.
\]
This prevents $u$ from $\lim\limits_{t\to T_{\max,0}^-}\norm{u(t)}_{L^\infty_{x,\omega}} = +\infty$ unless $T_{\max,0}=+\infty$. The solution $u$ is therefore globally defined in time.

\section{Proof of Proposition \ref{prop:approximate r=1}}\label{sec:approx}

First of all, since $u$ is a weak solution to \eqref{eq:general kinetic_0}-\eqref{eq:integration kernel_0}, $u^{(\eps)}$ is a weak solution to
\[
\partial_t u^{\eps} + \nabla_x\cdot\pa{c\omega u^{\eps}} = \frac{\sigma}{\eps} \Delta_\omega u^{\eps} +  \frac{\nu}{\eps}  \nabla_\omega \cdot \cro{u^\eps \mathbf{P}_{\omega^\bot}(\mathbf{F_0}[u])\pa{\frac{t}{\eps},\frac{x}{\eps},\omega}}.
\]
The result will hold true if 
\[
R_\eps(t,x,\omega) = \nabla_\omega \cdot \cro{u^\eps \mathbf{P}_{\omega^\bot}(\mathbf{F_0}[u])\pa{\frac{t}{\eps},\frac{x}{\eps},\omega}} - \nabla_\omega \cdot \cro{u^\eps \mathbf{P}_{\omega^\bot}(\mathbf{F_1}[u^\eps])}
\]
 satisfies the desired estimate.
\\Take a function $\phi$ in $C^\infty_c([0,+\infty)\times\R^d\times\S^{d-1})$ and, using integration by parts, compute
\begin{equation*}
\begin{split}
&\abs{\int_0^t\int_{\R^d\times\S^{d-1}}R_\eps(s,x,\omega)\phi(s,x,\omega)dsdxd\omega} 
\\&\quad\quad\quad= \abs{\int_0^t\int_{\R^d\times\S^{d-1}}u^\eps\nabla_\omega\phi \cdot \cro{\mathbf{P}_{\omega^\bot}(\mathbf{F_0}[u])\pa{\frac{s}{\eps},\frac{x}{\eps},\omega} - \mathbf{P}_{\omega^\bot}(\mathbf{F_1}[u^\eps])}dsdxd\omega}
\\&\quad\quad\quad \leq \int_0^t\int_{\R^d\times\S^{d-1}}u^\eps\abs{\nabla_\omega\phi} \abs{\mathbf{F_0}[u]\pa{\frac{s}{\eps},\frac{x}{\eps},\omega} - \mathbf{F_1}[u^\eps]}dsdxd\omega.
\end{split}
\end{equation*}
We thus study
\begin{eqnarray*}
&&\abs{\mathbf{F_0}[u]\pa{\frac{s}{\eps},\frac{x}{\eps},\omega} - \mathbf{F_1}[u^\eps])}
\\&& = \int_{\S^{d-1}}\cro{ \pa{\int_{\R^d}\mathbf{K}\pa{\frac{x}{\eps} -x_*, \omega,\omega_*} u^\eps\pa{s,\eps x_*,\omega_*}dx_*} - \mathbf{k}(\omega,\omega_*)u^\eps(s,x,\omega_*) }d\omega_*
\\&&= \int_{\S^{d-1}}\cro{ \pa{\int_{\R^d}\mathbf{K}\pa{x_*, \omega,\omega_*} u^\eps\pa{s,x-\eps x_*,\omega_*}dx_*} - \mathbf{k}(\omega,\omega_*)u^\eps(s,x,\omega_*) }d\omega_*
\\&& = \int_{\S^{d-1}}\int_{\R^d}\mathbf{K}\pa{x_*, \omega,\omega_*} \cro{u^\eps\pa{s,x - \eps x_*,\omega_*}-u^\eps(s,x,\omega_*)}dx_*.
\end{eqnarray*}

We emphasize that this is the first time where we actually use the property that $\mathbf{F_0}$ is a convolution operator in the spatial variable, see \eqref{eq:integration kernel_0}.
\\ Now, $\phi(t,x,\omega)$ and $\mathbf{K}(x_*,\omega, \omega_*)$ are of compact support in respectively $x$ and $x_*$, independently of $t$, $\omega$ and $\omega_*$. Let us so fix $R>0$ such that
\[
\mbox{Supp}(\phi(t,\cdot,\omega)) \subset B_x(0,R) \quad\mbox{and}\quad \mbox{Supp}(\mathbf{K}(\cdot,\omega,\omega_*)) \subset B_{x_*}(0,R).
\]
Thanks to the previous computations we obtain
\begin{equation*}
\begin{split}
&\abs{\int_0^t\int_{\R^d\times\S^{d-1}}R_\eps(s,x,\omega)\phi(s,x,\omega)dsdxd\omega} 
\\&\leq C_{\phi,u^\eps}\int_0^t\int_{\S^{d-1}\times B_x(0,R) \times B_{x_*}(0,R)}\abs{u^{\eps}(s,x-\eps x_*,\omega_*) - u^\eps(s,x,\omega_*)}dsdxdx_*d\omega d\omega_*.
\end{split}
\end{equation*}
where $C_{\phi,u^\eps} = \abs{S^{d-1}}\norm{\mathbf{K}}_{L^\infty_{x_*,\omega,\omega_*}}\norm{u^\eps}_{L^\infty_{t,x,\omega}}\norm{\phi}_{L^\infty_{t,x,\omega}}=O(\eps)$ since $u^\eps$ is $C^1$ by using the finite increments. 


\end{document}